\documentclass[12pt,reqno]{amsart}
\usepackage{latexsym}
\usepackage{amsbsy}
\usepackage{amssymb}
\usepackage{amsfonts}
\usepackage{amsmath}
\usepackage[latin1]{inputenc}
\usepackage{graphicx}
\usepackage{a4wide}
\allowdisplaybreaks

 \newtheorem{thm}{Theorem}[section]
 
 \newtheorem{lem}[thm]{Lemma}
 \newtheorem{prop}[thm]{Proposition}
 \newtheorem{cor}[thm]{Corollary}
 \newtheorem{rem}[thm]{Remark}

 \newcommand{\bthm}{\begin{thm}}
 \newcommand{\ethm}{\end{thm}}
 \newcommand{\bd}{\begin{defin}}
 \newcommand{\ed}{\end{defin}}
 \newcommand{\blem}{\begin{lem}}
 \newcommand{\elem}{\end{lem}}
 \newcommand{\bcor}{\begin{cor}}
 \newcommand{\ecor}{\end{cor}}
 \newcommand{\bprop}{\begin{prop}}
 \newcommand{\eprop}{\end{prop}}
 \newcommand{\brem}{\begin{rem} \rm}
 \newcommand{\erem}{\end{rem}}
 \newcommand{\bex}{\begin{ex} \rm}
 \newcommand{\eex}{\end{ex}}

 \newcommand{\beq}{\begin{equation}}
 \newcommand{\eeq}{\end{equation} }

 \newcommand{\bea}{\begin{eqnarray}}
 \newcommand{\eea}{\end{eqnarray}}

 \newcommand{\beas}{\begin{eqnarray*}}
 \newcommand{\eeas}{\end{eqnarray*}}

 \newcommand{\beqs}{\begin{equation*}}
 \newcommand{\eeqs}{\end{equation*}}

 \newcommand{\bi}{\begin{itemize}}
 \newcommand{\ei}{\end{itemize}}

 \newcommand{\ben}{\begin{enumerate}}
 \newcommand{\een}{\end{enumerate}}

 \newcommand{\ba}{\begin{array}}
 \newcommand{\ea}{\end{array}}

\newcommand{\NN}{\mathbb N}

\newcommand{\RR}{\mathbb R}

\newcommand{\DD}{\mathcal D}
\newcommand{\SSS}{\mathcal S}

\begin{document}

\title[Extension theorem of Whitney type for $\mathcal S(\mathbb{R}_+^d)$]{Extension theorem of Whitney type for $\mathcal S(\mathbb{R}_+^d)$ by the use of the Kernel Theorem}

\author{Smiljana Jak\v si\'c}
           \address{Faculty of Forestry, Belgrade University, Kneza Vi\v seslava 1, Belgrade, Serbia \\
              Tel.: +381-11-305398,
              Fax: +381-11-3053988}
           \email{smiljana.jaksic@sf.bg.ac.rs}

\author{Bojan Prangoski}
           \address{Faculty of Mechanical Engineering, University Ss.
           Cyril and Methodius, Karpos II bb, 1000 Skopje, Macedonia}
           \email{bprangoski@yahoo.com}

\begin{abstract}
We study the expansions of the elements in $\mathcal
S(\mathbb{R}_+^d)$ and $\mathcal{S}'(\mathbb{R}_+^d)$ with respect
to the Laguerre orthonormal basis, extending the result of M.
Guillemont-Teissier \cite{Mari} in the case $d=1$. As a
consequence, we obtain the Schwartz kernel theorem for
$\mathcal{S}(\mathbb{R}_+^d)$ and $\mathcal{S}'(\mathbb{R}_+^d)$ and the extension theorem of Whitney type for $\mathcal{S}(\mathbb{R}_+^d)$.
\end{abstract}

\maketitle

 \section{Introduction}

We denote by $\RR^d_+$ the set $(0,\infty)^d$ and by $\overline{\RR^d_+}$ its closure, i.e. $[0,\infty)^d$. We will
consider the space $\SSS(\mathbb{R}_+^d)$  which consists of all $f\in \mathcal{C}^{\infty}(\RR^d_+)$ such that all derivatives $D^pf$, $p\in\NN^d_0$, extend to continuous functions on $\overline{\RR^d_+}$ and
$$\sup_{x\in\mathbb{R}^d_+}x^k|D^pf(x)|<\infty\;,\forall k,p\in\mathbb{N}_0^d.$$
With this system of seminorms $\mathcal S(\mathbb{R}_+^d)$ becomes an $(F)$-space.

The results concerning the extension of a smooth function or a
function of class $\mathcal{C}^k$ out of some region and various
reformulation of such problems are called extension theorems of
Whitney type. One can see Whitney \cite{Whitney}, Seeley
\cite{Seeley} and Hörmander \cite[Theorem 2.3.6, p. 48]{Hermander}. Here we deal with a
problem of extension of a function from $\mathcal
S(\mathbb{R}_+^d)$ onto $\mathcal S(\mathbb{R}^d)$. Theorem
\ref{repofsupp} is the main result of the paper. For the purpose
of this theorem we prove the Schwartz kernel theorem for
$\mathcal{S}(\mathbb{R}_+^d)$ and $\mathcal{S}'(\mathbb{R}_+^d)$,
Theorem \ref{kernel thm}.

Recall, for $n=0,1,2...$ the functions
$$L_n(x)=\frac{e^x}{n!}\Big(\frac{d}{dx}\Big)^n(e^{-x}x^n),\quad x>0$$
are the Laguerre polynomials and
$\mathcal{L}_n(x)=L_n(x)e^{-\frac{x}{2}}$ are the Laguerre
functions; $\{\mathcal{L}_n(x)$, $n=0,1,...\}$ is an orthonormal
basis for $L^2(0,\infty)$ (\cite{Szego} p.108).

The problem of expanding the elements of $\mathcal
S'(\mathbb{R}_+)$ with respect to the Laguerre orthonormal basis
has been treated by  Guillemont-Teissier in \cite{Mari} and Duran
in \cite{D}:

\noindent If $T\in\mathcal S'(\mathbb{R}_+)$ and $a_n=\langle
T,\mathcal{L}_n(x)\rangle$ then
$T=\sum\limits_{n=0}^{\infty}a_n\mathcal L_n(x)$ and
$\{a_n\}_{n=0}^\infty$ decreases slowly. Conversely, if
$\{a_n\}_{n=0}^\infty$ decreases slowly, then there exists
$T\in\mathcal S'(\mathbb{R}_+)$ such that
$T=\sum\limits_{n=0}^{\infty}a_n\mathcal L_n(x)$.

The works \cite{SP1}, \cite{Zayed} and \cite{Zemanian} contain expansions of the same kind as in \cite{Mari} and \cite{D}.

The novelty of this paper is the extension of the results of
\cite{Mari} for the $d$ dimensional case. This leads to the
Schwartz kernel theorem (Theorem \ref{kernel thm}) which states
that there is one-to-one correspondence between elements from
$\mathcal{S}'(\mathbb{R}^{m+n}_+)$ in two sets of variables $x$
and $y$ and the continuous linear mappings of
$(\mathcal{S}(\mathbb{R}^m_+))_y$ into
$(\mathcal{S}'(\mathbb{R}^m_+))_x$. As a consequence of Theorem
\ref{repofsupp} we explain the convolution in $\SSS'(\RR^d_+)$ in
the last remark.

The plan of the paper is as follows. We recall in section $3$ some
properties of the Laguerre series and we prove the convergence of
the Laguerre series in $\mathcal{S}(\mathbb{R}_+^d)$ and
$\mathcal{S}'(\mathbb{R}_+^d)$. In section $4$ we state the
Schwartz's kernel theorem for $\mathcal{S}(\mathbb{R}_+^d)$ and we prove the extension theorem of Whitney type for $\mathcal{S}(\mathbb{R}_+^d)$.

 \section{Notation}

We use the standard multi-index notation. Given
$\alpha=(\alpha_1,...,\alpha_d)\in\mathbb{N}_0^d$, we write
$|\alpha|=\sum_{i=1}^d\alpha_i$, $x^{
\alpha}=(x_1,...,x_d)^{(\alpha_1,...,\alpha_d)}=\prod_{i=1}^dx_i^{\alpha_i}$,
$D^\alpha=\prod_{i=1}^d\frac{\partial^{\alpha_i}}{\partial
{x_i}^{\alpha_i}}$ for the partial derivative and $X^\alpha
f(x)=x^\alpha f(x)$ for the multiplication operator. For $x\in\RR^d$, $|x|$ stands for the standard Euclidean norm in $\RR^d$.

Let $s$ be the space of rapidly decreasing sequences

$$\{a_n\}_{n\in\mathbb{N}_0^d}\in s \Leftrightarrow\sum_{n\in\mathbb{N}_0^d}|a_n|^2n^{2k}<\infty,\qquad\forall k\in\mathbb{N}.$$

\noindent Then $s'$ stands for the strong dual of $s$, the space
of slowly increasing sequences

$$\{a_n\}_{n\in\mathbb{N}_0^d}\in s' \Leftrightarrow\sum_{n\in\mathbb{N}_0^d}|a_n|^2n^{-2k}<\infty,\qquad\exists k\in\mathbb{N}.$$

 \section{Laguerre series}

The $d$-dimensional Laguerre functions

$$\mathcal{L}_n(x)=\mathcal{L}_{n_1}(x_1)...\mathcal{L}_{n_d}(x_d)=\prod_{i=1}^d\mathcal{L}_{n_i}(x_i)$$

\noindent form an orthonormal basis for $L^2(\mathbb{R}^d_+)$ and are the eigenfunctions of the Laguerre operator
$E=\left(D_1(x_1D_1)-\frac{x_1}{4}\right)\ldots\left(D_d(x_dD_d)-\frac{x_d}{4}\right)$,
$E:\mathcal{S}(\mathbb{R}_+^d)\rightarrow\mathcal{S}(\mathbb{R}_+^d)$

$$\mathcal{L}_n(x)\rightarrow E(\mathcal{L}_n(x))=\prod_{i=1}^d-(n_i+\frac{1}{2})\mathcal{L}_n(x).$$

\noindent Notice that $E$ is a self-adjoint operator, i.e. \beas
\langle Ef,g\rangle=\langle f,Eg\rangle,\,\,\,
f,g\in\mbox{dom}(E)=\{f\in L^2(\mathbb{R}^d_+);\;Ef\in
L^2(\mathbb{R}^d_+)\}. \eeas For $f\in\mathcal{S}(\mathbb{R}^d_+)$
we define the $n$-th Laguerre coefficient by
$a_n=\int_{\mathbb{R}^d_+} f(x)\mathcal{L}_n(x)dx$. The Laguerre
series of the function $f\in\mathcal{S}(\mathbb{R}^d_+)$ is
$\sum_{n\in\mathbb{N}_0^d}a_n\mathcal{L}_n(x)$.

In \cite{Mari}, p. 547 the following bound on the one-dimensional
Laguerre functions is obtained:

$$\Big|x^k\Big(\frac{d}{dx}\Big)^p\mathcal{L}_n(x)\Big|\leq C_{p,k}(n+1)^{p+k},\;x\geq0,\;n,p,k\geq0.$$

\noindent Finding the bound on the $d$-dimensional Laguerre
functions involves not complicated calculation. Hence:

\begin{equation}\label{bound on Laguerre}
|x^kD^p\mathcal{L}_n(x)|\leq
C_{p,k}\prod_{i=1}^d(n_i+1)^{p_i+k_i},\;x\in\mathbb{R}^d_+,\;n,p,k\in\mathbb{N}_0^d.
\end{equation}

\subsection{Convergence of the Laguerre series in $\mathcal{S}(\mathbb{R}^d_+)$}

\begin{thm}\label{Thm Konvergencija S}
For $f\in\mathcal{S}(\mathbb{R}^d_+)$ let $a_n(f)=\int_{\mathbb{R}^d_+} f(x)\mathcal{L}_n(x)dx$. Then
$f=\sum_{n\in\mathbb{N}_0^d}a_n(f)\mathcal{L}_n$ and the series converges absolutely in $\mathcal{S}(\mathbb{R}^d_+)$. Moreover the mapping $\iota:\mathcal{S}(\mathbb{R}^d_+)\rightarrow s$, $\iota(f)=\{a_n(f)\}_{n\in\NN^d_0}$ is a topological isomorphism.
\end{thm}

\begin{proof}
For $f\in \mathcal{S}(\mathbb{R}^d_+)$
$$a_n(Ef)=\langle Ef,\mathcal{L}_n\rangle=\langle f,E(\mathcal{L}_n)\rangle=a_n(f)(-1)^d\prod_{i=1}^d\Big(n_i+\frac{1}{2}\Big).$$
\noindent Moreover,
\beas
a_n(E^pf)=a_n(f)\prod_{i=1}^d(-1)^{p_i}(n_i+\frac{1}{2})^{p_i}
\eeas
for any $p\in\mathbb{N}^d$. As
$E^pf\in\mathcal{S}(\mathbb{R}^d_+)\subset L^2(\mathbb{R}^d_+)$, we have

$$\sum_{n\in\mathbb{N}_0^d}|a_n(f)|^2\prod_{i=1}^d\Big(n_i+\frac{1}{2}\Big)^{2p_i}<\infty,\;\mbox{for every}\;p\in\mathbb{N}^d_0,$$ i.e. $\{a_n(f)\}_{n\in\NN^d_0}\in s$. Clearly $f=\sum_{n\in\mathbb{N}_0^d}a_n(f)\mathcal{L}_n$ as elements of $L^2(\mathbb{R}^d_+)$. By (\ref{bound on Laguerre}), we obtain
\begin{equation}\label{bound on Laguerre1}
\sum_{n\in\mathbb{N}_0^d}|x^kD^p(a_n(f)\mathcal{L}_n(x))|\leq C_{p,k}\sum_{n\in\mathbb{N}_0^d}|a_n(f)|\prod_{i=1}^d(n_i+1)^{p_i+k_i}<\infty
\end{equation}
which yields the absolute convergence of the series in $\mathcal{S}(\mathbb{R}^d_+)$.\\
\indent To prove that $\iota$ is topological isomorphism, first
observe that by the above consideration it is well defined and it is clearly an
injection. Let $\{a_n\}_{n\in\mathbb{N}_0^d}\in s$. Define
$f=\sum_{n\in\mathbb{N}_0^d}a_n\mathcal{L}_n\in
L^2(\mathbb{R}^d_+)$. Now (\ref{bound on Laguerre1}) proves that
this series converges in $\mathcal{S}(\mathbb{R}^d_+)$, hence
$f\in \mathcal{S}(\mathbb{R}^d_+)$. Thus $\iota$ is bijective.
Observe that, (\ref{bound on Laguerre1}) proves that $\iota^{-1}$
is continuous. Since $\mathcal{S}(\mathbb{R}^d_+)$ and $s$ are
$(F)$-spaces, the open mapping theorem proves that $\iota$ is
topological isomorphism.
\end{proof}

\subsection{Convergence of the Laguerre series in $\mathcal{S}'(\mathbb{R}^d_+)$}

\begin{thm}\label{Thm Konvergencija S'}
For $T\in\mathcal{S}'(\mathbb{R}^d_+)$, let $b_n(T)=\langle
T,\mathcal{L}_n\rangle$. Then
$T=\sum_{n\in\mathbb{N}_0^d}b_n(T)\mathcal{L}_n$ and
$\{b_n(T)\}_{n\in\mathbb{N}_0^d}\in s'$. The series converges
absolutely in $\mathcal{S}'(\mathbb{R}^d_+)$. Conversely, if
$\{b_n\}_{n\in\mathbb{N}_0^d}\in s'$, then there exists
$T\in\mathcal{S}'(\mathbb{R}^d_+)$ such that
$T=\sum_{n\in\mathbb{N}_0^d}b_n\mathcal{L}_n$. As a consequence,
$\mathcal{S}'(\mathbb{R}^d_+)$ is topologically isomorphic to
$s'$.
\end{thm}

\begin{proof}
Let $\{b_n\}_{n\in\mathbb{N}_0^d}\in s'$. There exists $k\in
\mathbb{N}$ such that $\sum_{n\in\mathbb{N}_0^d} |b_n|^2
(|n|+1)^{-2k}<\infty$. For a bounded subset $B$ of
$\mathcal{S}(\mathbb{R}^d_+)$, Theorem \ref{Thm Konvergencija S}
implies that there exists $C>0$ such that
$$\sum_{n\in\mathbb{N}_0^d}|a_n(f)|^2(|n|+1)^{2k}\leq C,\,\, \forall f\in B,$$
where we denote $\{a_n(f)\}_{n\in\mathbb{N}_0^d}=\iota(f)$. Observe that for arbitrary $q\in \mathbb{N}$ we have
$$\sum_{|n|\leq q}\sup_{f\in B}|\langle b_n\mathcal{L}_n,f\rangle|\leq \sup_{f\in B}
\sum_{n\in\mathbb{N}_0^d}\sum_{m\in\mathbb{N}_0^d}|\langle b_n\mathcal{L}_n,a_m(f)\mathcal{L}_m\rangle|=
\sup_{f\in B} \sum_{n\in\mathbb{N}_0^d} |b_n||a_n(f)|\leq C',$$
i.e.
\beas
\sum_{n\in\mathbb{N}_0^d}\sup_{f\in B}|\langle b_n\mathcal{L}_n,f\rangle|<\infty,
\eeas
hence $\sum_{n\in\mathbb{N}_0^d}b_n\mathcal{L}_n$ converges absolutely in $\mathcal{S}'(\mathbb{R}^d_+)$.\\
\indent Let $T\in\mathcal{S}'(\mathbb{R}^d_+)$. Theorem \ref{Thm
Konvergencija S} implies that ${}^t\iota: s'\rightarrow
\mathcal{S}'(\mathbb{R}^d_+)$ is an isomorphism (${}^t\iota$
denotes the transpose of $\iota$). Now, one easily verifies that
$({}^t\iota)^{-1}T=\{b_n\}_{n\in\mathbb{N}_0^d}$, where
$b_n(T)=\langle T,\mathcal{L}_n\rangle$. Observe that for $f\in
\mathcal{S}(\mathbb{R}^d_+)$
$$\langle T, f\rangle =\sum_{n\in\mathbb{N}_0^d}a_n(f)\langle T,\mathcal{L}_n\rangle =\sum_{n\in\mathbb{N}_0^d}a_n(f)b_n(T)=\left\langle \sum_{n\in\mathbb{N}_0^d}b_n(T)\mathcal{L}_n,f\right\rangle,$$
i.e. $T=\sum_{n\in\mathbb{N}_0^d}b_n(T)\mathcal{L}_n$.
\end{proof}

 \section{Kernel theorem}

The completions of the tensor product are denoted by
$\hat{\otimes}_\epsilon$ and $\hat{\otimes}_\pi$ with respect to
$\epsilon$ and $\pi$ topologies. If they are equal we drop the
subindex.

\begin{prop}\label{Prop nuc}
The spaces $\SSS(\RR^d_+)$ and $\SSS'(\RR^d_+)$ are nuclear.
\end{prop}

\begin{proof}
Since $s$ is nuclear Theorem \ref{Thm Konvergencija S} implies that $\SSS(\RR^d_+)$ is also nuclear. Now $\SSS'(\RR^d_+)$ is nuclear as the strong dual of a nuclear $(F)$-space.
\end{proof}

\begin{thm}\label{kernel thm}
The following canonical isomorphisms hold:

$$\mathcal{S}(\mathbb{R}^m_+)\hat{\otimes}\mathcal{S}(\mathbb{R}^n_+)\cong
\mathcal{S}(\mathbb{R}^{m+n}_+),\,\,\, \mathcal{S}'(\mathbb{R}^m_+)\hat{\otimes}\mathcal{S}'(\mathbb{R}^n_+)\cong
\mathcal{S}'(\mathbb{R}^{m+n}_+).$$

\end{thm}

\begin{proof}
The second isomorphism follows from the first since $\SSS(\RR^d_+)$ is a nuclear $(F)$-space. Thus it is enough to prove the first isomorphism.\\
Step 1: From Theorem \ref{Thm Konvergencija S} follows that
$\mathcal{S}(\mathbb{R}^m_+)\otimes\mathcal{S}(\mathbb{R}^n_+)$ is
dense in $\mathcal{S}(\mathbb{R}^{m+n}_+).$ It suffices to show
that the latter induces on the former the topology $\pi=\epsilon$
(the $\pi$ and the $\epsilon$ topologies are the same because
$\mathcal{S}(\mathbb{R}^d_+)$ is nuclear). Since the bilinear
mapping $(f,g)\mapsto f\otimes g$ of
$\mathcal{S}(\mathbb{R}^m_+)\times \mathcal{S}(\mathbb{R}^n_+)$
into $\mathcal{S}(\mathbb{R}^{m+n}_+)$ is separately continuous it
follows that it is continuous ($\mathcal{S}(\mathbb{R}^m_+)$ and $\mathcal{S}(\mathbb{R}^n_+)$ are $(F)$-spaces). The continuity of this bilinear
mapping proves that the inclusion
$\mathcal{S}(\mathbb{R}^m_+)\otimes_\pi
\mathcal{S}(\mathbb{R}^n_+)\rightarrow
\mathcal{S}(\mathbb{R}^{m+n}_+)$ is continuous, hence the topology
$\pi$ is stronger than the induced one from
$\mathcal{S}(\mathbb{R}^{m+n}_+)$ onto
$\mathcal{S}(\mathbb{R}^m_+)\otimes \mathcal{S}(\mathbb{R}^n_+)$.

Step 2: Let $A'$ and $B'$ be equicontinuous subsets of
$\mathcal{S}'(\mathbb{R}^m_+)$ and $\mathcal{S}'(\mathbb{R}^n_+)$,
respectively. There exist $C>0$ and $j,l\in \mathbb{N}$ such that
such that
$$\sup_{T\in A'}|\langle T,\varphi\rangle|\leq
C\|\varphi\|_{j,l}\quad\mbox{and}\quad\sup_{F\in B'}|\langle
F,\psi\rangle|\leq C\|\psi\|_{j,l},$$
\noindent where
\begin{equation}\label{semin}
\|f\|_{j,l}=\sup_{\substack{|k|\leq j\\ |p|\leq l}}\sup_{x\in\mathbb{R}^d_+}|x^kD^pf(x)|<\infty.
\end{equation}
\noindent For all $T\in A'$ and $F\in B'$ we have
\begin{eqnarray*}
& & |\langle T_x\otimes F_y,\chi(x,y)\rangle| = |\langle
F_y,\langle T_x,\chi(x,y)\rangle\rangle|\leq
C\sup_{\substack{|k|\leq j\\ |p|\leq
l}}\sup_{y\in\mathbb{R}^n_+}|y^{k}\langle T_x,
D^{p}_y\chi(x,y)\rangle|\\ & & \leq C^2\sup_{\substack{|k|\leq j\\
|p|\leq l}}\sup_{\substack{|k'|\leq j\\ |p'|\leq
l}}\sup_{\substack{x\in\mathbb{R}^m_+\\
y\in\mathbb{R}^n_+}}|x^{k'} y^{k}D^{p'}_xD^{p}_y\chi(x,y)|\\ & &
\leq C^2\|\chi(x,y)\|_{(k',k),(p',p)},\,\,\forall \chi\in
\mathcal{S}(\mathbb{R}^m_+)\otimes \mathcal{S}(\mathbb{R}^n_+).
\end{eqnarray*}
\noindent It follows that the $\epsilon$ topology on
$\mathcal{S}(\mathbb{R}^m_+)\otimes \mathcal{S}(\mathbb{R}^n_+)$
is weaker than the induced one from
$\mathcal{S}(\mathbb{R}^{m+n}_+)$.
\end{proof}

As a consequence of this theorem we have the following important

\begin{thm}\label{repofsupp}
The restriction mapping $f\mapsto f_{|\RR^d_+}$, $\SSS(\RR^d)\rightarrow \SSS(\RR^d_+)$ is a topological homomorphism onto.\\
\indent The space $\SSS(\RR^d_+)$ is topologically isomorphic to the quotient space $\SSS(\RR^d)/N$, where $N=\{f\in\SSS(\RR^d)|\, \mathrm{supp}\, f\subseteq \RR^d\backslash \RR^d_+\}$. Consequently, $\SSS'(\RR^d_+)$ can be identified with the closed subspace of $\SSS'(\RR^d)$ which consists of all tempered distributions with support in $\overline{\RR^d_+}$.
\end{thm}

\begin{proof} Obviously, the restriction mapping $f\mapsto f_{|\RR^d_+}$, $\SSS(\RR^d)\rightarrow \SSS(\RR^d_+)$ is continuous.
We prove its surjectivity by induction on $d$. For clarity, denote the $d$-dimensional restriction by $R_d$. For $d=1$, the
surjectivity of $R_1$ is proved in \cite[p. 168]{D}. Assume that $R_d$ is surjective. By the open mapping theorem, $R_d$ and
$R_1$ are topological homomorphisms onto since all the underlying spaces are $(F)$-spaces. By the above theorem
$R_d\hat{\otimes}_{\pi} R_1$ is continuous mapping from $\SSS(\RR^{d+1})$ to
$\SSS(\RR^{d+1}_+)$ ($\SSS(\RR^d)\hat{\otimes}\SSS(\RR)\cong \SSS(\RR^{d+1})$ by the Schwartz kernel theorem).
Clearly $R_d\hat{\otimes}_{\pi} R_1=R_{d+1}$. As $\SSS(\RR^{d+1})$ and $\SSS(\RR^{d+1}_+)$ are $(F)$-spaces
\cite[Theorem 7, p. 189]{kothe2} implies that $R_{d+1}$ is also surjective.\\
\indent The surjectivity of the restriction mapping together with the open mapping theorem implies that it is homomorphism.
Clearly $N$ is closed subspace of $\SSS(\RR^d)$ and $\mathrm{ker}\, R_d=N$. Thus $R_d$ induces natural topological isomorphism
between $\SSS(\RR^d)/N$ and $\SSS(\RR^d_+)$. Hence $\left(\SSS(\RR^d)/N\right)'_b$ is topologically isomorphic to
$\SSS'(\RR^d_+)$ (the index $b$ stands for the strong dual topology). Since $\SSS(\RR^d)$ is an $(FS)$-space,
\cite[Theorem A.6.5, p. 255]{morimoto} implies that $\left(\SSS(\RR^d)/N\right)'_b$ is topologically isomorphic to the closed
subspace
\beas
N^{\perp}=\{T\in\SSS'(\RR^d)|\, \langle T, f\rangle=0,\,\, \forall f\in N\}
\eeas
of $\SSS'(\RR^d)$ which is exactly the
subspace of all tempered distributions with support in $\overline{\RR^d_+}$.
\end{proof}

Given $f,g\in\SSS'(\RR^d_+)$, Theorem \ref{repofsupp} implies that we can consider them as elements of $\SSS'(\RR^d)$ with
support in $\overline{\RR^d_+}$. Now, one easily verifies that for each $\varphi\in\SSS(\RR^d)$,
\beas
(f(x)\otimes g(y))\varphi(x+y)\in\DD'_{L^1}(\RR^{2d}),
\eeas
hence the $\SSS'$-convolution of $f$ and $g$ exists (see \cite[p. 26]{shiraishi}). Also, if
$\mathrm{supp}\,\varphi\cap\overline{\RR^d_+}=\emptyset$, then $(f(x)\otimes g(y))\varphi(x+y)=0$, hence $\mathrm{supp}\, f*g
\subseteq \overline{\RR^d_+}$, i.e. $f*g\in\SSS'(\RR^d_+)$. Thus
\beas
\langle f\ast g,\varphi\rangle=\langle f(x)\otimes g(y),\varphi(x+y)\rangle,\,\, \varphi\in \SSS(\mathbb{R}^d_+)
\eeas
(observe that
the function $\varphi^{\Delta}(x,y)=\varphi(x+y)$ is an element of $\SSS(\RR^{2d}_+)$).\\

\begin{rem}(\cite{D}, Remark 3.7 for d=1) Let us show that
$\mathcal{S}'(\mathbb{R}_+^d)$ is a convolution algebra.
Given $f,g\in\mathcal{S}'(\mathbb{R}_+^d)$, we compute the $n$-th
Laguerre coefficient of $f\ast g$ if $a_n=\langle
f,\mathcal{L}_n\rangle$ and $b_n=\langle g,\mathcal{L}_n\rangle$
then

$$\langle f\ast g, \mathcal{L}_n(t)\rangle=\langle f(x)\otimes
g(y),\mathcal{L}_n(x+y)\rangle.$$

\noindent Now,
$L_n^1(x+y)=\sum_{k=0}^nL_{n-k}(x)L_k(y)\;\mbox{and}\;L_n(t)=L_n^1(t)-L_{n-1}^1(t)$
(see \cite{Ed}, p. 192), where
$L_n^1(x)=\sum_{k=0}^n\binom{n+1}{n-k}((-x)^k/k!)$. In order to
simplify the proof, we consider the case $d=2$. Then

\begin{eqnarray*}
& & \langle f\ast g,\mathcal{L}_n(t)\rangle=\langle f(x)\otimes
g(y),\prod_{i=1}^2(\mathcal{L}_{n_i}^1(x_i+y_i)-\mathcal{L}_{n_i-1}^1(x_i+y_i))\rangle\\
& & =\langle f(x)\otimes
g(y),\prod_{i=1}^2\Big(\sum_{k_i=0}^{n_i}\mathcal{L}_{n_i-k_i}(x_i)\mathcal{L}_{k_i}(y_i)
-\sum_{k_i=0}^{n_i-1}\mathcal{L}_{n_i-k_i-1}(x_i)\mathcal{L}_{k_i}(y_i)\Big)\rangle\\
& & =\langle
f(x)g(y),\sum_{k\leq(n_1,n_2)}\mathcal{L}_{(n_1,n_2)-k}(x)\mathcal{L}_k(y)-\sum_{k\leq(n_1-1,n_2)}\mathcal{L}_{(n_1-1,n_2)-k}(x)\mathcal{L}_k(y)\\
& &
\quad-\sum_{k\leq(n_1,n_2-1)}\mathcal{L}_{(n_1,n_2-1)-k}(x)\mathcal{L}_k(y)+\sum_{k\leq(n_1-1,n_2-1)}\mathcal{L}_{(n_1-1,n_2-1)-k}(x)\mathcal{L}_k(y)\rangle\\
& & =
\sum_{k\leq(n_1,n_2)}a_{(n_1,n_2)-k}b_k-\sum_{k\leq(n_1-1,n_2)}a_{(n_1-1,n_2)-k}b_k\\
& &
\quad-\sum_{k\leq(n_1,n_2-1)}a_{(n_1,n_2-1)-k}b_k+\sum_{k\leq(n_1-1,n_2-1)}a_{(n_1-1,n_2-1)-k}b_k,
\end{eqnarray*}

\noindent where $a_n$ or $b_n$ equals zero if some component of
the subindex $n $ is less than zero. It is easy to verify that if
$(a_n)_{n\in\mathbb{N}^2}\in s'$ and $(b_n)_{n\in\mathbb{N}^2}\in
s'$ then $\langle f\ast g,\mathcal{L}_n(t)\rangle\in s'$.

\end{rem}

 \subsection*{Acknowledgement}

The paper was supported by the projects {\it Modelling and
harmonic analysis methods and PDEs with singularities}, No. 174024
financed by the Ministry of Science, Republic of Serbia.



\begin{thebibliography}{00}

\bibitem{D}
A. J. Duran, Laguerre expansions of Tempered Distributions and
Generalized Functions, Journal of Mathematical Analysis and
Applications 150 (1990), 166-180.

\bibitem{Ed}
A. Erdelyi, Higher Transcedentals Function, Vol. 2, McGraw-Hill,
New York, 1953.

\bibitem{Hermander} L. Hörmander, The analysis of linear partial differential operators. I. Distribution theory and Fourier analysis, Springer-Verlag, 1990.

\bibitem{Mari} M. Guillemont-Teissier, D\'{e}veloppements des distributions en s\'{e}ries de fonctions orthogonales. S\'{e}ries de Legendre et de Laguerre,
Annali dellaq Scuola Normale Superiore di Pisa (3) 25 (1971),
519-573.

\bibitem{morimoto} M.~Morimoto, An introduction to Sato's hyperfunctions. Vol. 129. American Mathematical Soc., 1993.

\bibitem{kothe2} G. K\"{o}the, Topological vector spaces II, Vol.II, Springer-Verlag, New York Inc., 1979.

\bibitem{SP1} S.~Pilipovic, On the Laguerre expansions of generalized functions, C. R. Math. Rep. Acad. Sci. Canada 11 (1989), no. 1, 23-27

\bibitem{Seeley} R. T. Seeley, Extension of C8 functions defined in a half space, Proc. Amer. Math. Soc. 15 (1964), 625-626.

\bibitem{shiraishi} R.~Shiraishi, On the definition of convolutions for distributions, J. Sci. Hiroshima Univ. Ser. A 23 (1959),
19-32.

\bibitem{Szego} G. Szego, Orthogonal polynomials, Am. Math. Soc. Collequium, 1959.

\bibitem{Whitney} H. Whitney, Analytic extensions of functions defined in closed sets, Transactions of the American Mathematical Society 36 (1934), 63-89.

\bibitem{Zayed} A. I. Zayed, Laguerre series as boundary values. SIAM J. Math. Anal. 13 (1982), no. 2, 263-279

\bibitem{Zemanian} A. H. Zemanian, Generalized Integral Transformations, Intersci. , New York, 1968.

\end{thebibliography}
\end{document}